\documentclass{amsart}
\usepackage{amssymb}
\usepackage{mathrsfs}
\usepackage{graphicx}
\usepackage{comment}

\newcommand{\be}{\begin{equation}}
\newcommand{\ee}{\end{equation}}
\newcommand{\ba}{\begin{align}}
\newcommand{\ea}{\end{align}}

\newtheorem{theorem}{Theorem}[section]
\newtheorem{lemma}[theorem]{Lemma}

\newtheorem*{proposition*}{Proposition}

\newtheorem*{theorem*}{Theorem}
\newtheorem*{corollary*}{Corollary}
\newtheorem*{cor*}{Corollary}	
	
\theoremstyle{definition}

\newtheorem{claim}[theorem]{Claim}

{\begin{list}{}{%
\settowidth{\labelwidth}{\textsf{{\it #1.}}}%
\setlength{\labelsep}{4mm}%
\setlength{\leftmargin}{\labelwidth}%
\addtolength{\leftmargin}{\labelsep}%
}}%
{\end{list}}

\allowdisplaybreaks

\def\beq{\begin{equation}}\def\enq{\end{equation}}

{\begin{list}{}{%
\settowidth{\labelwidth}{\textsf{{\it #1.}}}%
\setlength{\labelsep}{2mm}%
\setlength{\leftmargin}{\labelwidth}%
\addtolength{\leftmargin}{\labelsep}%
\addtolength{\leftmargin}{4mm}%
\setlength{\itemsep}{6pt}%
\setlength{\listparindent}{0pt}%
\setlength{\topsep}{3pt}%
}}%

\usepackage{color}
\usepackage[normalem]{ulem} 
\usepackage{soul}

\usepackage{xcolor}

\begin{document}

\title[Elementary Proofs of Recent Congruences]{Elementary Proofs of Recent Congruences for Overpartitions Wherein Non-Overlined Parts are Not Divisible by 6}

\author[B. Paudel]{Bishnu Paudel}

\author[J. Sellers]{James A. Sellers}

\author[H. Wang]{Haiyang Wang}

\address{Mathematics and Statistics Department\\
         University of Minnesota Duluth\\
         Duluth, MN 55812, USA}
\email{bpaudel@d.umn.edu, jsellers@d.umn.edu, wan02600@d.umn.edu}

\begin{abstract}
We define $\overline{R_l^*}(n)$ as the number of overpartitions of $n$ in which non-overlined parts are not divisible by $l$. In a recent work,  Nath, Saikia, and the second author established several families of congruences for $\overline{R_l^*}(n)$, with particular focus on the cases $l=6$ and $l=8$. In the concluding remarks of their paper, they conjectured that $\overline{R_6^*}(n)$ satisfies an infinite family of congruences modulo $128$. In this paper, we confirm their conjectures using elementary methods. Additionally, we provide elementary proofs of two congruences for $\overline{R_6^*}(n)$ previously proven via the machinery of modular forms by Alanazi, Munagi, and Saikia.
\end{abstract} 

\maketitle

\section{Introduction}
A {\it partition} of a positive integer $n$ is a finite non--increasing sequence of positive integers $(\lambda_1, \lambda_2, \dots, \lambda_k)$ whose sum equals $n$.  The integers $\lambda_1, \lambda_2, \dots, \lambda_k$ are called the {\it parts} of the partition.  As an example, the number of partitions of the integer $n=4$ is 5, and the partitions in question are
$$(4), \ \ \ (3,1), \ \ \ (2,2), \ \ \ (2,1,1), \ \ \ (1,1,1,1).$$
More information about integer partitions can be found in \cite{And1976, AE}.

One generalization of an integer partition is an {\it overpartition} of $n$ \cite{CL} which is a partition of $n$ wherein the first occurrence of a part may be overlined.  As an example, there are 14 overpartitions of  $n=4$:  
\begin{gather*}
(4), \ \ \ (\overline{4}), \ \ \ (3,1), \ \ \ (\overline{3},1), \ \ \ (3,\overline{1}), \ \ \ (\overline{3},\overline{1}), \ \ \ (2,2), \ \ \  (\overline{2},2),\\
(2,1,1), \ \ \  (\overline{2},1,1), \ \ \  (2,\overline{1},1), \ \ \  (\overline{2},\overline{1},1), \ \ \ (1,1,1,1), \ \ \ (\overline{1},1,1,1).
\end{gather*}
The number of overpartitions of $n$ is often denoted $\overline{p}(n)$; from the above we see that $\overline{p}(4) = 14$. 

Since the work of Corteel and Lovejoy \cite{CL}, a variety of restricted overpartition functions have been defined and analyzed.  As an example, Alanazi, Alenazi, Keith, and Munagi \cite{AAKM22} considered the family of functions $\overline{R_{\ell}^*}(n)$ which counts the number of overpartitions of weight $n$ wherein non-overlined parts are not allowed to be divisible by $\ell$ while there are no restrictions on the overlined parts.  For example, there are 12 overpartitions counted by $\overline{R_3^*}(4)$:  
\begin{gather*}
(4), \ \ \ (\overline{4}), \ \ \  (\overline{3},1), \ \ \ (\overline{3},\overline{1}), \ \ \ (2,2), \ \ \  (\overline{2},2),\\
(2,1,1), \ \ \  (\overline{2},1,1), \ \ \  (2,\overline{1},1), \ \ \  (\overline{2},\overline{1},1), \ \ \ (1,1,1,1), \ \ \ (\overline{1},1,1,1).
\end{gather*}
One can readily see that two overpartitions counted by $\overline{p}(4)$, namely $(3,1)$ and $(3,\overline{1})$, do not appear in the list above.  This is true because they contain a non-overlined part which is divisible by $\ell=3$.

In \cite{AAKM22}, Alanazi et al. proved a number of congruence properties satisfied by the functions $\overline{R_{\ell}^*}(n)$ which, for each $\ell$, satisfies the generating function identity 
\begin{equation*}
\sum_{n=0}^\infty\overline{R_{\ell}^*}(n)q^n =  \frac{f_2 f_{\ell}}{f_1^2}
\end{equation*}
where $$f_k = \prod_{m=1}^\infty (1-q^{km}).$$  
Subsequently, additional work on this family of functions has been completed; see \cite{Alanazi-Munagi-Saikia, nath-saikia-sellers, SP25, Sel25} for examples of such work.   

Our goal in this brief paper is to utilize truly elementary means to prove two different sets of results for the function $\overline{R_{6}^*}(n)$.  
First, we note the following theorem which combines the statements of two conjectures that recently appeared in the work of Nath, Saikia, and the second author \cite{nath-saikia-sellers}.


\begin{theorem}
\label{conj}
\cite[Conjecture 8.1 and Conjecture 8.2]{nath-saikia-sellers}\label{conj1}
    For all $n\geq0$ and $k\geq 0$, we have 
    \begin{equation}\label{R6183...128}
    \overline{R_6^*}\left(18\cdot3^{2k+1}n+\frac{153\cdot3^{2k}-1}{4}\right)\equiv0\pmod{128}.
    \end{equation}
\end{theorem}
Next, we mention a pair of congruences given by Alanazi, Munagi, and Saikia \cite[Theorem 4.4]{Alanazi-Munagi-Saikia}.  It is important to note that the authors proved these properties via an automated approach which relies on the machinery of modular forms; our goal here is to provide a classical proof for each of these congruences.
\begin{theorem}\label{TR36} For $n\geq0$, we have
	\begin{align}
      	\label{627n1164} \overline{R_6^*}(27n+11) &\equiv 0 \pmod{64}, \\
		\label{681n4724} \overline{R_6^*}(81n+47) &\equiv 0 \pmod{24}.
	\end{align}
\end{theorem}

In order to prove Theorems \ref{conj} and \ref{TR36}, we will need a few foundational results which already appear in the literature.  We gather all of the necessary results here.  We begin with a well--known identity of Jacobi.  
\begin{lemma}[Jacobi]
\label{lem:jacobi}
 We have
\begin{equation}
\label{f13} 
f_1^3=\sum_{m\geq0}(-1)^m(2m+1)q^{m(m+1)/2}.
\end{equation}
\end{lemma}
\begin{proof}
See Hirschhorn \cite[Equation (1.7.1)]{H2017}.
\end{proof}
Next, we share a pair of 2--dissection identities that will be useful in our work below.  
\begin{lemma}
 \begin{align}
\label{f1/f33}\frac{f_1}{f_3^3}&=\frac{f_2f_4^2f_{12}^2}{f_6^7}-q\frac{f_2^3f_{12}^6}{f_4^2f_6^9},\\
   \label{f13/f3}\frac{f_1^3}{f_3}&=\frac{f_4^3}{f_{12}}-3q\frac{f_2^2f_{12}^3}{f_4f_6^2}.
   \end{align}
   \end{lemma}
\begin{proof}
Equations~\eqref{f1/f33} and~\eqref{f13/f3} correspond to~(29) and~(30), respectively, in~\cite[Lemma~1]{SS20}.
\end{proof}
In analogous fashion, we also require several 3--dissection results which will be used in our generating function manipulations below.  
   \begin{lemma} We have
   \begin{align}
   \label{f12/f2}\frac{f_1^2}{f_2}&=\frac{f_9^2}{f_{18}}-2q\frac{f_3f_{18}^2}{f_6f_9},\\
   \label{f22/f1_d}\frac{f_2^2}{f_1}&=\frac{f_6f_9^2}{f_3f_{18}}+q\frac{f_{18}^2}{f_9},\\
   \label{f2/f12}\frac{f_2}{f_1^2}&=\frac{f_6^4f_9^6}{f_3^8f_{18}^3}+2q\frac{f_6^3f_9^3}{f_3^7}+4q^2\frac{f_6^2f_{18}^3}{f_3^6},\\
    \label{f1f2}f_1f_2&=\frac{f_6f_9^4}{f_3f_{18}^2}-qf_9f_{18}-2q^2\frac{f_3f_{18}^4}{f_6f_9^2},\\
   \label{f13d}f_1^3&=\frac{f_6f_9^6}{f_3f_{18}^3}-3qf_9^3+4q^3\frac{f_3^2f_{18}^6}{f_6^2f_9^3},\\
  \label{1/f13}\frac{1}{f_1^3}&=\frac{f_6^2f_9^{15}}{f_3^{14}f_{18}^6}+3q\frac{f_6f_9^{12}}{f_3^{13}f_{18}^3}+9q^2\frac{f_9^9}{f_3^{12}}+8q^3\frac{f_9^6f_{18}^3}{f_3^{11}f_6}+12q^4\frac{f_9^3f_{18}^6}{f_3^{10}f_6^2}\\
   \notag&+16q^6\frac{f_{18}^{12}}{f_3^8f_6^4f_9^3}.
   \end{align}
\end{lemma}
\begin{proof}
Equations~\eqref{f12/f2} and~\eqref{f22/f1_d} appear as~(14.3.2) and~(14.3.3) in~\cite{H2017}, respectively. Identity~\eqref{f2/f12} was proven in~\cite{HS05}, and~\cite{HS14} contains a proof of~\eqref{f1f2}. The identities~\eqref{f13d} and~\eqref{1/f13} can be found in~\cite[Lemma~3]{NA24}.
\end{proof}
Lastly, we need the following well--known fact which basically follows from the Binomial Theorem and divisibility properties of certain binomial coefficients.  
\begin{lemma}
    For a prime $p$ and positive integers $k$ and $l$,
    \begin{equation}\label{modp^k}
        f_l^{p^k}\equiv f_{lp}^{p^{k-1}}\pmod{p^k}.
    \end{equation}
\end{lemma}


\section{Proof of theorem \ref{conj}}
\label{sec:proof_of_first_theorem}

We begin by recalling the generating function
\begin{align*}
    \sum_{n=0}^\infty\overline{R_6^*}(n)q^n &=  \frac{f_2 f_6}{f_1^2}\\
    &=\left(\frac{f_6^4 f_9^6}{f_3^8 f_{18}^3} + 2q\,\frac{f_6^3 f_9^3}{f_3^7} + 4q^2\,\frac{f_6^2 f_{18}^3}{f_3^6}\right) f_6 \quad \text{(thanks to \eqref{f2/f12})}.
\end{align*}
Extracting the terms in which the exponents of $q$ are of the form $3n+2$, dividing both sides by $q^2$, and then replacing $q^3$ by $q$, we get

\begin{align}\label{R63n2}
	\sum_{n=0}^\infty \overline{R_6^*}(3n+2) q^n &= 4\frac{f_2^3 f_6^3}{f_1^6}\\
   \notag &=4\frac{f_2^3 f_6^3}{f_1^{32}}f_1^{26}\\
   \notag &\equiv4\frac{f_2^3f_6^3}{f_2^{16}}f_1^{26} \pmod{128} \quad \text{(thanks to \eqref{modp^k})}\\
   \notag &=4f_6^3\left(\frac{f_1^2}{f_2}\right)^{13}\\
   \notag & = 4f_6^3\left(\frac{f_9^2}{f_{18}}-2q\frac{f_3f_{18}^2}{f_6f_9}\right)^{13}.
\end{align}
Observing that $4(a-2b)^{13}\equiv 4a^{13}+24a^{12}b+96a^{11}b^2+64a^{10}b^3+64a^9b^4\pmod{128}$, extracting the terms in which the exponents of $q$ are of the form $3n$, we get 
\begin{align*}
   \sum_{n=0}^\infty \overline{R_6^*}(9n+2) q^{3n} &\equiv f_6^3\left(4\frac{f_9^{26}}{f_{18}^{13}}+64\frac{f_9^{20}}{f_{18}^{10}}\cdot q^3\frac{f_3^3f_{18}^6}{f_6^3f_9^3}\right)\pmod{128}.
\end{align*}
Replacing $q^3$ by $q$ gives

\begin{align}
   \label{R69n2128} \sum_{n=0}^\infty \overline{R_6^*}(9n+2) q^{n} \equiv \sum_{n=0}^{\infty}T_1(9n+2)q^n+\sum_{n=0}^{\infty}T_2(9n+2)q^n \pmod{128},
\end{align}
where 
\begin{align}
   \label{T1} \sum_{n=0}^{\infty}T_1(9n+2)q^n&=4\frac{f_2^3f_3^{26}}{f_6^{13}},\\
  \label{T2} \sum_{n=0}^{\infty} T_2(9n+2)q^n&= 64q\frac{f_1^3f_{3}^{17}}{f_6^4}\equiv 64qf_1^3f_3^9\pmod{128}\quad \text{(thanks to \eqref{modp^k})}.
\end{align}

For given $n\geq0$ and $k\geq0$, setting
$$l_{n,k}:=18\cdot3^{2k+1}n+\frac{153\cdot3^{2k}-1}{4},$$
we have $l_{n,k}\equiv2\pmod9$. To show that $\overline{R_6^*}(l_{n,k})\equiv0\pmod{128}$, by \eqref{R69n2128}, it suffices to prove the following two congruences 
\begin{align}
\label{T1l0128}T_1(l_{n,k})&\equiv0\pmod{128},\\
\label{T2l0128}T_2(l_{n,k})&\equiv0\pmod{128}.
\end{align}

\subsection*{Proof of \eqref{T1l0128}}

Using \eqref{f13} in \eqref{T1}, we get 
\begin{align}
   \label{T11} \sum_{n=0}^{\infty}T_1(9n+2)q^n=4\frac{f_3^{26}}{f_6^{13}}\left(\sum_{m\geq0}(-1)^m(2m+1)q^{m(m+1)}\right).
\end{align}
We now check whether $m(m+1)+3k=6n+4$ for some $m,n$ and $k$. Equivalently, $(2m+1)^2+12k=24n+17$. This is not possible since $5$ is a quadratic nonresidue modulo $12$. Thus, the right-hand side of \eqref{T11} does not contain terms in which the exponents of $q$ are of the form $6n+4$, and hence 
\begin{equation*}
\label{T16n4}
T_1(54n+38)\equiv0\pmod{128},
\end{equation*}
which implies that $T_1(l_{n,k})\equiv0\pmod{128}$ when $k=0$.

In order to show $T_1(l_{n,k})\equiv0\pmod{128}$ for $k\geq1$, we first establish the following claim. 
\begin{claim} \label{ClaimT1}
For $k\geq1$, we have 
\begin{align}
  \label{T1C}  \sum_{n=0}^{\infty}T_1\left(3^{2k+2}n+\frac{3^{2k+2}-1}{4}\right)q^n \equiv \pm32qf_1^3f_3^9-(16a+12)\cfrac{f_1^8f_3^{18}}{f_2f_6^9}\pmod{128} 
\end{align}
for some integer $a$. Here, $\pm$ indicates that ~\eqref{T1C} takes either the $+$ or the $-$ sign, not both at once.
\end{claim}
\begin{proof}[Proof of Claim \ref{ClaimT1}] We prove this by induction on $k$. Applying \eqref{modp^k} to \eqref{T1}, we get
\begin{align*}
   \sum_{n=0}^{\infty}T_1(9n+2)q^n&\equiv4\frac{f_2^3f_6^3}{f_3^6} \pmod{128}\\
   &=4 \left(\frac{f_{12}f_{18}^6}{f_6f_{36}^3}-3q^2f_{18}^3+4q^6\frac{f_6^2f_{36}^6}{f_{12}^2f_{18}^3}\right)\frac{f_6^3}{f_3^6} \quad \text{(thanks to \eqref{f13d})}.
\end{align*}
Extracting the terms which contain the form $q^{3n+2}$, dividing by $q^2$, and replacing $q^3$ by $q$, we get 
\begin{align*}
    \sum_{n=0}^{\infty}T_1(27n+20)q^n&\equiv-12f_6^3\left(\frac{f_2}{f_1^2}\right)^3 \pmod{128}\\
    &= -12f_6^3\left(\frac{f_6^4f_9^6}{f_3^8f_{18}^3}+2q\frac{f_6^3f_9^3}{f_3^7}+4q^2\frac{f_6^2f_{18}^3}{f_3^6}\right)^3 \quad \text{(using \eqref{f2/f12})}.
\end{align*}
Extracting the terms that contain exponents of $q$ of the form $3n$ gives

\begin{align*}
    \sum_{n=0}^{\infty}T_1(81n+20)q^{3n}\equiv-12f_6^3\left(\frac{f_6^{12}f_9^{18}}{f_3^{24}f_{18}^9}+56q^3\frac{f_6^9f_9^9}{f_3^{21}}\right) \pmod{128}.
\end{align*}
We replace $q^{3}$ by $q$ to obtain
\begin{align*}
\notag\sum_{n=0}^{\infty}T_1\left(3^4n+\frac{3^4-1}{4}\right)q^n&\equiv-32q\frac{f_2^{12}f_3^9}{f_1^{21}}-12\frac{f_2^{15}f_3^{18}}{f_1^{24}f_{6}^9} \pmod{128}\\
&\equiv-32qf_1^{3}f_3^9-12\frac{f_2^{15}f_3^{18}}{f_1^{24}f_{6}^9} \pmod{128},
 \end{align*}
where the last congruence follows applying \eqref{modp^k}. This establishes the claim for $k=1$.

Suppose that \eqref{T1C} holds for a fixed $k$. Then, we show that \eqref{T1C} holds for $k+1$. From \eqref{T1C}, we have
\begin{align*}
   &\sum_{n=0}^{\infty}T_1\left(3^{2k+2}n+\frac{3^{2k+2}-1}{4}\right)q^n\\
    &\equiv \pm32qf_1^3f_3^9-(16a+12)\cfrac{f_1^8f_3^{18}}{f_2f_6^9}\pmod{128}\\
    &=\pm32qf_1^3f_3^9-(16a+12)(f_1^3)^2\frac{f_1^2}{f_2}\frac{f_3^{18}}{f_6^9}\\
    &=\pm32qf_3^9\left(\frac{f_6f_9^6}{f_3f_{18}^3}-3qf_9^3+4q^3\frac{f_3^2f_{18}^6}{f_6^2f_9^3}\right)-(16a+12)\left(\frac{f_6f_9^6}{f_3f_{18}^3}-3qf_9^3+4q^3\frac{f_3^2f_{18}^6}{f_6^2f_9^3}\right)^2\\
    &\times\left(\frac{f_9^2}{f_{18}}-2q\frac{f_3f_{18}^2}{f_6f_9}\right)\frac{f_3^{18}}{f_6^9} \quad \text{(using \eqref{f13d} and \eqref{f12/f2}}).
\end{align*}
We extract the terms in which the exponents of $q$ are of the form $3n+2$ to obtain
\begin{align*}
&\sum_{n=0}^{\infty}T_1\left(3^{2k+2}(3n+2)+\frac{3^{2k+2}-1}{4}\right)q^{3n+2} \\
 &\equiv \pm32 q^2f_3^9f_9^3-(16a+12)\left(21q^2\frac{f_3^{18}f_9^8}{f_6^9f_{18}}+48q^5\frac{f_3^{21}f_{18}^8}{f_6^{12}f_9}\right) 
 \pmod{128}.
\end{align*}
Dividing by $q^2$ and replacing $q^{3}$ by $q$ yields
\begin{align*}
&\sum_{n=0}^{\infty}T_1\left(3^{2k+3}n+\frac{3^{2k+4}-1}{4}\right)q^n\\
 &\equiv \pm32f_1^9f_3^3-21(16a+12)\frac{f_1^{18}f_3^8}{f_2^9f_{6}}+64q\frac{f_1^{21}f_{6}^8}{f_2^{12}f_3}
 \pmod{128}\\
 &\equiv\pm32f_1^9f_3^3-21(16a+12)\left(\frac{f_1^{2}}{f_2}\right)^9\frac{f_3^8}{f_6}+64q\frac{f_{6}^8}{f_1^3f_3} \pmod{128} \quad \text{(thanks to \eqref{modp^k})}\\
 &\equiv\pm32f_3^3\left(\frac{f_6f_9^6}{f_3f_{18}^3}-3qf_9^3\right)^3-21(16a+12)\left(\frac{f_9^2}{f_{18}}-2q\frac{f_3f_{18}^2}{f_6f_9}\right)^9\frac{f_3^8}{f_6}\\
 &+64q\frac{f_6^8}{f_3}\left(\frac{f_6^2f_9^{15}}{f_3^{14}f_{18}^6}+3q\frac{f_6f_9^{12}}{f_3^{13}f_{18}^3}+9q^2\frac{f_9^9}{f_3^{12}}\right) \pmod{128} \quad \text{(using \eqref{f12/f2}, \eqref{f13d}, \eqref{1/f13})}.
\end{align*}
We observe that $12(x-2y)^9\equiv12(x^9-18x^8y+16x^7y^2)\pmod{128}$. So, extracting the terms that contain the form $q^{3n}$, we get 
\begin{align*}
    &\sum_{n=0}^{\infty}T_1\left(3^{2k+3}(3n)+\frac{3^{2k+4}-1}{4}\right)q^{3n}\\
    &\equiv \pm32\left(\frac{f_6^3f_9^{18}}{f_{18}^9}-27q^3f_3^3f_9^9\right)-21(16a+12)\frac{f_3^8f_9^{18}}{f_6f_{18}^9}+64q^3\frac{f_6^8f_9^9}{f_3^{13}} \pmod{128}.
\end{align*}
We replace $q^3$ by $q$ to obtain
\begin{align*}
&\sum_{n=0}^{\infty}T_1\left(3^{2k+4}n+\frac{3^{2k+4}-1}{4}\right)q^{n}\\
    &\equiv \pm32\frac{f_2^3f_3^{18}}{f_{6}^9}\pm32qf_1^3f_3^9-21(16a+12)\frac{f_1^8f_3^{18}}{f_2f_{6}^9}+64q\frac{f_2^8f_3^9}{f_1^{13}} \pmod{128}\\
    &\equiv \pm32\frac{f_1^8f_3^{18}}{f_2f_6^9}\pm32qf_1^3f_3^9-21(16a+12)\frac{f_1^8f_3^{18}}{f_2f_{6}^9}+64qf_1^3f_3^9 \pmod{128} \quad \text{(by \eqref{modp^k})}\\
    &=\mp32qf_1^3f_3^9-\big(21(16a+12)\mp32\big)\frac{f_1^8f_3^{18}}{f_2f_6^9}\pmod{128}.
\end{align*}
Note that $21(16a+12)+32\equiv16(5a+1)+12 \pmod{128}$ and $21(16a+12)-32\equiv16(5a+5)+12 \pmod{128}$. This completes both the induction and proof of the claim.
\end{proof}

From \eqref{T1C}, for $k\geq1$, we have 
\begin{align*}
 &\sum_{n=0}^{\infty}T_1\left(3^{2k+2}n+\frac{3^{2k+2}-1}{4}\right)q^n\\
    &\equiv \pm32qf_1^3f_3^9-(16a+12)(f_1^3)^2\frac{f_1^2}{f_2}\frac{f_3^{18}}{f_6^9}\pmod{128}\\
    &=\pm32qf_3^9\left(\frac{f_6f_9^6}{f_3f_{18}^3}-3qf_9^3+4q^3\frac{f_3^2f_{18}^6}{f_6^2f_9^3}\right)-(16a+12)\left(\frac{f_6f_9^6}{f_3f_{18}^3}-3qf_9^3+4q^3\frac{f_3^2f_{18}^6}{f_6^2f_9^3}\right)^2\\
    &\times\left(\frac{f_9^2}{f_{18}}-2q\frac{f_3f_{18}^2}{f_6f_9}\right)\frac{f_3^{18}}{f_6^9} \quad \text{(using \eqref{f13d} and \eqref{f12/f2}}).
\end{align*}
Extracting the terms that contain the form $q^{3n+1}$, we get 
\begin{align*}
 &\sum_{n=0}^{\infty}T_1\left(3^{2k+2}(3n+1)+\frac{3^{2k+2}-1}{4}\right)q^{3n+1}\\
 &\equiv \pm32q\frac{f_3^8f_6f_9^6}{f_{18}^3}-(16a+12)\left(-40q^4\frac{f_3^{20}f_9^2f_{18}^5}{f_6^{11}}-8q\frac{f_3^{17}f_9^{11}}{f_6^8f_{18}^4}\right) \pmod{128}.
\end{align*}
Dividing by $q$ and replacing $q^3$ by $q$ gives 
\begin{align*}
&\sum_{n=0}^{\infty}T_1\left(3^{2k+3}n+\frac{5\cdot3^{2k+2}-1}{4}\right)q^n\\
&\equiv \pm32\frac{f_1^8f_2f_3^6}{f_6^3}-32q\frac{f_1^{20}f_3^2f_{6}^5}{f_2^{11}}-32\frac{f_1^{17}f_3^{11}}{f_2^8f_{6}^4}\pmod{128}\\
&\equiv 32f_3^2\left(\pm\frac{f_2^5}{f_6}-q\frac{f_6^5}{f_2}-\frac{f_1}{f_3^3}f_6^2\right) \pmod{128} \quad \text{thanks to \eqref{modp^k}}\\
&=32f_3^2\left(\pm\frac{f_2^5}{f_6}-q\frac{f_6^5}{f_2}-\left(\frac{f_2f_4^2f_{12}^2}{f_6^7}-q\frac{f_2^3f_{12}^6}{f_4^2f_6^9}\right)f_6^2\right) \pmod{128} \quad \text{(using \eqref{f1/f33})}\\
&\equiv 32f_3^2\left(\pm\frac{f_2^5}{f_6}-\frac{f_2^5}{f_6}\right) \pmod{128} \quad \text{(thanks to \eqref{modp^k})}\\
&\equiv\begin{cases} 0 \pmod{128}&\text{when taking positive sign},\\
-64f_2^5 \pmod{128} &\text{when taking negative sign and applying \eqref{modp^k}}.
\end{cases}
\end{align*}
Observe that the right-hand side of the last congruence contains no terms that contain odd powers of $q$. Thus, for $n\geq0$ and $k\geq1$, we have 
\begin{align*}
T_1\left(3^{2k+3}(2n+1)+\frac{5\cdot3^{2k+2}-1}{4}\right)\equiv 0\pmod{128},
\end{align*}
where
$$3^{2k+3}(2n+1)+\frac{5\cdot3^{2k+2}-1}{4}=18\cdot3^{2k+1}n+\frac{153\cdot3^{2k}-1}{4}=l_{n,k}.$$

\subsection*{Proof of \eqref{T2l0128}}
We now establish the following claim for $T_2$.
\begin{claim}\label{T2C} For $n\geq0$ and $k\geq0$, we have
\begin{align}
   \label{T2final} \sum_{n=0}^{\infty} T_2\left(2\cdot3^{2k+2}n+\frac{3^{2k+2}-1}{4}\right)q^{n} &\equiv
    \lambda 64f_1^3+64qf_3^3f_6^3 \pmod{128} \vspace{0.2cm},
\end{align}
where $\lambda=0 \text{ or } 1$.
\end{claim}
\begin{proof}[Proof of Claim \eqref{T2C}]
We prove the claim by induction on $k$. Applying \eqref{modp^k} to \eqref{T2}, we get 
\begin{align*}
  \sum_{n=0}^{\infty} T_2(9n+2)q^n&\equiv 64q\frac{f_1^3}{f_3}f_6^5 \pmod{128}\\
  &= 64q\left(\frac{f_4^3}{f_{12}}-3q\frac{f_2^2f_{12}^3}{f_4f_6^2}\right)f_6^5 \quad \text{(using \eqref{f13/f3})}.
\end{align*}
We extract the terms that contain even powers of $q$ and then replace $q^2$ by $q$ to obtain 
\begin{align*}
   \notag \sum_{n=0}^{\infty} T_2\left(2\cdot3^{2}n+\frac{3^{2}-1}{4}\right)q^{n}&\equiv64q\frac{f_1^2f_3^3f_6^3}{f_2} \pmod{128}\\
  &\equiv 64qf_3^3f_6^3 \pmod{128} \quad \text{(thanks to \eqref{modp^k})}.
\end{align*}
This establishes the claim for $k=0$.

Suppose that \eqref{T2final} holds for a fixed $k$. We show that it also holds for $k+1$. From \eqref{T2final}, we have
\begin{align*}
   \sum_{n=0}^{\infty} T_2\left(2\cdot3^{2k+2}n+\frac{3^{2k+2}-1}{4}\right)q^{n} &\equiv
    \lambda64f_1^3+64qf_3^3f_6^3 \pmod{128}\\
    &=\lambda64\left(\frac{f_6f_9^6}{f_3f_{18}^3}-3qf_9^3+4q^3\frac{f_3^2f_{18}^6}{f_6^2f_9^3}\right)+64qf_3^3f_6^3,
\end{align*}where the last equality follows using \eqref{f13d}.
Extracting the terms in which the exponents of $q$ are of the form $3n+1$, dividing by $q$ and replacing $q^3$ by $q$, we get
\begin{align*}
   \sum_{n=0}^{\infty} T_2\left(2\cdot3^{2k+3}n+\frac{3^{2k+4}-1}{4}\right)q^{n}&\equiv
    \lambda64f_3^3+64f_1^3f_2^3 \pmod{128}\\
    &=\lambda64 f_3^3+64\left(\frac{f_6f_9^4}{f_3f_{18}^2}-qf_9f_{18}-2q^2\frac{f_3f_{18}^4}{f_6f_9^2}\right)^3,
\end{align*}
using \eqref{f1f2}. Observe that $64(a-b-2c)^3\equiv 64(a^3+a^2b+ab^2+b^3)\pmod{128}$. We extract the terms that contain exponents of $q$ of the form $3n$ and replace $q^3$ by $q$ to get 
\begin{align*}
   \sum_{n=0}^{\infty} T_2\left(2\cdot3^{2k+4}n+\frac{3^{2k+4}-1}{4}\right)q^{n} &\equiv \lambda 64 f_1^3+64\left(\frac{f_2^3f_3^{12}}{f_1^3f_{6}^6}+qf_3^3f_{6}^3\right) \pmod{128}\\
   &\equiv \lambda 64 f_1^3+64\left(f_1^3+qf_3^3f_{6}^3\right) \pmod{128}\\
   & \quad \text{(using \eqref{modp^k}}) \\
   \notag &\equiv \begin{cases} 64qf_3^3f_6^3 \pmod{128} &\text{if $\lambda=1$}, \\ 
   64f_1^3+64qf_3^3f_6^3 \pmod{128} &\text{if $\lambda=0$}.
   \end{cases}
\end{align*}
This shows that \eqref{T2final} holds for $k+1$ and completes the proof of the claim.
\end{proof}

Since, from \eqref{f13d}, 
$$f_1^3=\frac{f_6f_9^6}{f_3f_{18}^3}-3qf_9^3+4q^3\frac{f_3^2f_{18}^6}{f_6^2f_9^3},$$
the right-hand sides in \eqref{T2final} do not contain terms in which the exponents of $q$ are of the form $3n+2$. Therefore, for all $n\geq0$ and $k\geq0$, we have
\begin{align*}
  T_2\left(2\cdot3^{2k+2}(3n+2)+\frac{3^{2k+2}-1}{4}\right)
  &\equiv0 \pmod{128},
\end{align*}
where 
$$2\cdot3^{2k+2}(3n+2)+\frac{3^{2k+2}-1}{4}=18\cdot3^{2k+1}n+\frac{153\cdot3^{2k}-1}{4}=l_{n,k}.$$ \qed

\section{Proof of Theorem \ref{TR36}}
\label{sec:Proof_of_Second_Theorem}
We close this work by quickly providing elementary proofs of the congruences in Theorem \ref{TR36}.  These rely on our generating function manipulations above, and follow from a straightforward analysis of the dissections in question.  
\subsection*{Proof of $\eqref{627n1164}$}
From \eqref{R63n2}, we have 
\begin{align*}
	\sum_{n=0}^\infty \overline{R_6^*}(3n+2) q^n &=4\frac{f_1^{10}f_2^3f_6^3}{f_1^{16}}\\
    &\equiv 4 \frac{f_6^3 f_1^{10}}{f_2^5} \pmod{64} \quad \text{(applying \eqref{modp^k})} \\
	&= 4 f_6^3 \left( \frac{f_9^2}{f_{18}} - 2q\,\frac{f_3 f_{18}^2}{f_6 f_9} \right)^5 \quad \text{(thanks to $\eqref{f12/f2}$)}.
\end{align*}
As $4(a-2b)^5\equiv 4(a^5+6a^4b+8a^3b^2)\pmod{64}$, extracting the terms with exponents divisible by $3$ gives
\begin{align*}
	\sum_{n=0}^\infty \overline{R_6^*}(9n+2) q^{3n}
	&\equiv 4 f_6^3 \frac{f_9^{10}}{f_{18}^5} \pmod{64}.
\end{align*}
Replacing $q^3$ by $q$ yields
\begin{align*}
	\sum_{n=0}^\infty \overline{R_6^*}(9n+2) q^n &\equiv 4 f_2^3 \frac{f_3^{10}}{f_6^5} \pmod{64} \\
	&= 4 \left( \sum_{m \geq 0} (-1)^m (2m+1) q^{m(m+1)} \right) \frac{f_3^{10}}{f_6^5} \quad\text{(using \eqref{f13})}.
\end{align*}
The proof will be completed by showing that there exist no integers $m$ and $n$ satisfying
\[
m(m+1)=3n+1,
\]
or,
\[
(2m+1)^2=12n+5.
\]
Since $5$ is not a quadratic residue modulo $12$, no such integers exist.

\subsection*{Proof of $\eqref{681n4724}$}
Again from \eqref{R63n2}, we have 
\begin{align*}
	\sum_{n=0}^\infty \overline{R_6^*}(3n+2) q^n &=4\frac{f_2^3f_6^3}{f_1^{6}}\\
    &\equiv 4 \frac{f_2^3f_6^3 }{f_2^3} \pmod{8} \quad \text{(applying \eqref{modp^k})} \\
	&= 4 f_6^3.
\end{align*}
By extracting the terms of the form $q^{3n}$ and replacing $q^3$ by $q$, we obtain
\begin{align*}
	\sum_{n=0}^\infty \overline{R_6^*}(9n+2) q^n &\equiv 4f_2^3 \pmod{8} \\
	&= \sum_{m\ge 0} (-1)^m (2m+1)q^{m(m+1)} \quad \text{(thanks to \eqref{f13})}.
\end{align*}
We claim that there exist no integers $m$ and $n$ satisfying
\[
9n+5 = m(m+1),
\]
or equivalently,
\[
(2m+1)^2=36n+21.
\]
Since $21$ is not a quadratic residue modulo $36$, no such integers exist. It follows that
\begin{equation}\label{681n478}
	\overline{R_6^*}(81n+47) \equiv 0 \pmod{8}.
\end{equation}

Next, we apply \eqref{modp^k} to \eqref{R63n2} to  deduce that
\[
\sum_{n=0}^\infty \overline{R_6^*}(3n+2) q^n \equiv 4\frac{f_6^4}{f_3^2} \pmod{3}.
\]
Then,
\begin{align*}
	\sum_{n=0}^\infty \overline{R_6^*}(9n+2) q^n &\equiv 4\frac{f_2^4}{f_1^2} \pmod{3} \\
	&= 4\left(\frac{f_6f_9^2}{f_3f_{18}}+q\,\frac{f_{18}^2}{f_9}\right)^2 \quad \text{(using $\eqref{f22/f1_d}$)}.
\end{align*}
Extracting the terms with exponents of $q$ are of the form ${3n+2}$, dividing by $q^2$, and replacing $q^3$ by $q$ gives
\[
\sum_{n=0}^\infty \overline{R_6^*}(27n+20) q^n \equiv 4\frac{f_6^4}{f_3^2} \pmod{3}.
\]
Since the resulting series is expressed in terms of $q^3$, and therefore cannot contain any terms of the form $q^{3n+1}$, we conclude that
\begin{equation}\label{681n473}
	\overline{R_6^*}(27(3n+1)+20)=\overline{R_6^*}(81n+47) \equiv 0 \pmod{3}.
\end{equation}
Combining \eqref{681n478} and \eqref{681n473} completes the proof of $\eqref{681n4724}$. \qed

As we close, it is worth noting that the proof above can be modified in straightforward fashion to prove that, for all $n\geq 0$, 
$$
\overline{R_6^*}(27(3n+2)+20)=\overline{R_6^*}(81n+74) \equiv 0 \pmod{24},
$$
a result which was not mentioned in \cite{Alanazi-Munagi-Saikia}.  

\bibliographystyle{plain}
\nocite{*}
\bibliography{Overpartitions6-regular.bib}

\end{document}